\documentclass[12pt]{amsart}
\usepackage{pgfplots}
\usepackage{float}
\usepackage[utf8]{inputenc}
\usepackage{graphicx}
\usepackage{color}
\usepackage{mdwlist}
\usepackage{soul,xcolor}
\usepackage{ulem}
\allowdisplaybreaks

\usepackage{cleveref}

\newcommand{\beq}{\begin{eqnarray*}}
\newcommand{\feq}{\end{eqnarray*}}
\newcommand{\beqn}{\begin{eqnarray}}
\newcommand{\feqn}{\end{eqnarray}}

\newtheorem{theorem}{Theorem}[section]
\newtheorem{lemma}[theorem]{Lemma}

\theoremstyle{definition}

\theoremstyle{remark}
\newtheorem{remark}[theorem]{Remark}
\numberwithin{equation}{section}

\setstcolor{red}
\newtheorem*{theorem*}{Theorem}

\begin{document}
\title[Critical thresholds in Euler-Poisson systems]{Critical thresholds in one dimensional damped Euler-Poisson systems}

\author{Manas Bhatnagar and Hailiang Liu}
\address{Department of Mathematics, Iowa State University, Ames, Iowa 50010}
\email{manasb@iastate.edu}
\email{hliu@iastate.edu} 
\keywords{Euler-Poisson system, critical threshold, global regularity, shock formation}
\subjclass{Primary, 35L65; Secondary, 35L67} 

\begin{abstract} This paper is concerned with the critical threshold phenomenon for one dimensional damped, pressureless Euler-Poisson equations with electric force induced by a constant background, originally studied in [S. Engelberg and H. Liu and E. Tadmor, Indiana Univ. Math. J., 50:109--157, 2001]. A simple transformation is used to linearize the characteristic system of equations, which allows us to study the geometrical structure of critical threshold curves for three damping cases: overdamped, underdamped and borderline damped through phase plane analysis.  We also derive the explicit form of these critical curves. 
These sharp results state that if the initial data is within the threshold region, the solution will remain smooth for all time, otherwise it will have a finite time breakdown. Finally, we apply these general results to identify critical thresholds for a non-local system subjected to initial data on the whole line.
\end{abstract}

\maketitle

\section{Introduction} 
\label{prob} 
It is well known that  the finite-time breakdown of the systems of Euler equations for compressible flows is generic  in the sense that finite-time shock formation occurs for all but a ``small" set of initial data.  For pairs of conservation laws, Lax \cite{La64} showed that the $C^1$-smoothness of solutions can be lost  unless its two Riemann invariants are nondecreasing.  On the other hand, with the Poisson forcing the  system of Euler--Poisson equations  admits a ``large" set of initial configurations which yield global smooth solutions, see, e.g. \cite{ELT01, LT02a,  LT03}.  
The Euler-Poisson system of equations is used to model various phenomena,  ranging from plasma physics to applications in semiconductors.
Physically, we would desire to know as to whether the concerned particles aggregate or the smooth density profile exists forever.  Indeed, for a class of Euler--Poisson equations, the question addressed in \cite{ELT01}  is whether there is a critical threshold for the initial data such that the persistence of the $C^1$ solution regularity depends only on crossing such a critical threshold.   For example, for system of Euler--Poisson equations with only electric force,
\begin{align*}
& \rho_t +(\rho u)_x=0,\\
&  (\rho u)_t +(\rho u^2)_x=-k \rho \phi_x,\\
& -\phi_{xx}=\rho -c,
\end{align*}
subject to initial data $(\rho_0>0,  u_0(x))$, \cite{ELT01} has shown that for repulsive force $k>0$ it admits a global smooth solution if and only if 
$$
u_{0x}(x) > -\sqrt{2k \rho_0(x)} \quad \text{if} \quad c=0,
$$
and 
$$
|u_{0x}(x)| < \sqrt{k(2\rho_0(x)-c)} \quad \text{if} \quad c>0.
$$
These two critical thresholds indicate that with the background charge, the solutions of the above system will be oscillatory, hence initial slope cannot be too big either. 
The zero background case when augmented with the usual $\gamma$-law pressure, is shown by Tadmor and Wei \cite{TW08} to still admit global solutions for a large class of initial data identified by an intrinsic critical threshold. The non-zero background case with  pressure is different.  Using special energy techniques with proper normal form of transformations, the authors in \cite{GHZ17} have shown that smooth solutions with small amplitude persist forever with no shock formation in the case of cubic law of pressure. 

In this paper we revisit the one dimensional pressureless, damped  Euler--Poisson system with potential induced by a constant background, \begin{subequations}
\begin{align}
\label{EPMain}
\begin{aligned}
&\rho_t + (\rho u)_x = 0, \\
&u_t + uu_x = -k\phi_x -\nu u, \\
&-\phi_{xx} = \rho -c,\\
\end{aligned}
\end{align}
subject to initial conditions,
\begin{align}
\label{EPMain2}
\begin{aligned}
&\rho (x,0) = \rho_0 (x) > 0,\qquad \rho_0 \in C^1 (\mathbb{R}),\\
&u(x,0) = u_0 (x),\qquad u_0 \in C^1 (\mathbb{R}),\\
\end{aligned}
\end{align}
\end{subequations}
where $c>0$ is the constant background, $\nu >0$ is the damping coefficient,  and parameter $k$ signifies the property of the underlying forcing, repulsive if $k>0$ and attractive if $k<0$. We consider only repulsive force between particles and hence, $k>0$. Here, we also need the neutrality condition
\[
\int_{-\infty}^\infty (\rho_0(\xi)-c)\, d\xi = 0,
\]
which is conserved for all time if $\rho u$ vanishes at far fields. Therefore, we have a fixed background charge density of $c$ and an equal amount of movable charge, $\rho(x, t)$.

The main objective of our revisit  to this problem is to introduce alternative tools, instead of the use of flow map techniques in \cite{ELT01}, to identify the critical thresholds for (\ref{EPMain}).  We hope these tools can be useful for the study of critical threshold phenomena in other problems of similar nature. 
More precisely, we want to get an explicit characterization of the critical threshold curve as a function of initial density and velocity slope for three different cases:
\begin{enumerate}
\item
$\nu >2\sqrt{kc} $,  strong damping,
\item
$\nu < 2\sqrt{kc}$,  weak damping,  and
\item
$\nu = 2\sqrt{kc}$,  borderline damping.
\end{enumerate}
We are able to recover the results in \cite{ELT01} for weak damping case, and obtain a sharp critical threshold for strong damping case, for which only a sufficient condition was identified in \cite{ELT01}, plus a sharp critical threshold for the borderline case.

We present two methods in analysis, each gives the critical threshold curve for all three cases but by different techniques. The initial step in both methods is to transform a non-linear system of equations into a linear system and then analyze the obtained system. The first method is more rudimentary and involves explicit solution techniques of linear differential equations with constant coefficients. The second method involves vector field analysis. 

On the solution behavior of Euler--Poisson equations  there is a considerable amount of literature available. Consult  \cite{En96, WC98} for nonexistence results and singularity formation;   \cite{CW96, WW06} for global existence of weak solutions with geometrical symmetry; \cite{Gu98, GMP13, GP11} for global existence for 3-D irrotational flow, \cite{MN95} for isentropic case, and \cite{PRV95} for isothermal case. Smooth irrotational solutions for the two dimensional Euler--Poisson system are constructed independently in \cite{IP13, LW14}. See also \cite{Ja12, JLZ14} for related results on two dimensional case. The question of critical thresholds in multi-D Euler-Poisson systems remains largely open; we refer to \cite{LT03} for sharp conditions on global regularity vs finite time breakdown for the 2-D restricted Euler--Poisson system, and \cite{LT02a} for sufficient conditions on finite time breakdown for the general n-dimensional restricted Euler--Poisson systems. A relative complete analysis of critical thresholds in 3-D restricted Euler--Poisson systems is given in \cite{Lee2} for both attractive and repulsive forcing.   

As a direct benefit of  our present results, we illustrate how to  apply them to an interesting system in the context of biological aggregation:
\begin{subequations}\label{cc}
\begin{align}
& \rho_t + (\rho u)_x = 0, \quad (x,t)\in \mathbb{R}\times (0,\infty), \\
& u_t + uu_x + u = -\partial W\star\rho, 
\end{align}
\end{subequations} 
where 
\[W(x) = -|x|+\frac{|x|^2}{2},
\]
subject to initial conditions
\begin{align}
\label{ccid}
\begin{aligned}
& \rho (x,0) = \rho_0 (x) > 0,\qquad \rho_0 \in C^1(\mathbb{R}),\\
& u(x,0) = u_0 (x),\qquad u_0 \in C^1(\mathbb{R}).
\end{aligned}
\end{align}
Instead of electric force governed by the Poisson equation, here non-local interactions between particles are modeled by Newtonian attractive forces.  System (\ref{cc}) has been formally derived from interacting particle systems in collective dynamics; see e.g. \cite{CDMBC07}, \cite{CDP09}, 
and kinetic equations for collective behavior can be derived rigorously from particle systems via the mean-field limit, see \cite{CFTV10, CCR11}, 
and the references therein.

When initial data is compactly supported, belonging to space $(H^2(U),H^3(U))$, where $U\subset\mathbb{R}$ has compact support, the critical thresholds for this problem have been established in \cite{CCZ} by flow map techniques. 
{We should point out that in \cite[Remark 3.1]{CCZ} there is an explanation of what additional assumptions need to be made for the initial data so that the result still holds for $U=\mathbb{R}$}.  However, the analysis of the local existence result for classical solutions in \cite[Appendix A]{CCZ} does not seem to be applicable directly to the setting when initial data is defined on the whole line. Hence, we present a new local existence theorem of classical solutions and give a self-contained proof using a different approach, in which some control of solution behavior at far fields is essential.  In addition, the $C^1$ class of initial data we consider is larger than the $H^2\times H^3$ class for $(\rho_0, u_0)$.  {With this local existence theory, our results obtained for (\ref{EPMain}) when applied to (\ref{cc}) lead  to Theorem 5.2 and Theorem 5.3. To our best knowledge, the geometrical structure of the critical threshold curves for (\ref{cc}) given in Theorem 5.3 is new. The explicit thresholds in Theorem 5.2 are essentially the same as those in \cite{CCZ}, although here the initial data is defined on the whole line.  
} 

A related model is the one-dimensional Euler--alignment system which has a non-local velocity alignment force (such force becomes the linear damping when the alignment force is localized), for such model thresholds for global regularity vs finite time breakdown were analyzed in \cite{TT14}.  Such result  was further improved in \cite{CCTT16} by closing the gap between lower and upper thresholds. When both linear damping and nonlocal interaction forces are present, sharp critical thresholds were obtained in \cite{CCZ} for a special system (\ref{cc}) with smooth, compactly supported initial data.

The rest of this paper is organized as follows. In Section 2, we state the main results and introduce the key transformation as a preparation for the analysis carried out in Sections 3 and 4.  In Section 3, we prove our main results, providing sharp critical thresholds for initial configurations which yield either global smooth solution or finite time breakdown. In Section 4, we give dynamic representation of the critical threshold curve in each case.  Finally, in Section 5 we apply our obtained results  to identify the critical thresholds for (\ref{cc}). The proof of the needed local wellposedness result is deferred into Appendix A.  

\section{Preliminaries and main results}

The threshold analysis to be carried out is the a priori estimate on smooth solutions as long as they exist.  For the one-dimensional
Euler-Poisson problem, local existence of smooth solutions was long known, it can be justified by using the characteristic method in the pressureless case.  We only state the result here.

\begin{theorem}
\label{local}
$($\textbf{Local existence}$)$ If $\rho_0 \in C^1$ and  $u_0 \in C^1$ , then there exists $T>0$, depending on the initial data, such that the initial value problem (\ref{EPMain}),  (\ref{EPMain2}) admits a unique solution $(\rho, u) \in C^{1} ([0,T)\times \mathbb{R}).$  Moreover,  if the maximum life span $T^* < \infty$, then
$$
\lim_{t \rightarrow T^*-} \partial_x u(t, x^*) =-\infty
$$
for some $x^* \in \mathbb{R}$.
\end{theorem}
We proceed to derive the characteristic system which is essentially used to  in our critical threshold analysis. Differentiate the second equation in (\ref{EPMain}) with respect to $x$,  and set $u_x := d$ to obtain:
	\begin{subequations}\label{nonlin}
	\begin{align}
	& \rho' + \rho d = 0,  \\
	& d' + d^2 + \nu d = k(\rho -c ),
	\end{align}
	\end{subequations}
	where we have used the Poisson equation in (\ref{EPMain})  for $\phi$, and 
	\[
	\{\}' = \frac{\partial}{\partial t} + u\frac{\partial}{\partial x},
	\]
	 denotes the differentiation along the particle path,
	$$
	\Gamma=\{(x, t)|\;  x'(t)=u(x(t), t), x(0)=\alpha \in \mathbb{R}\}.
	$$
Here, we employ the method of characteristics to convert the PDE system (\ref{EPMain}) to ODE system (\ref{nonlin}) along the particle path which is fixed for a fixed value of the parameter $\alpha$. Consequently, the initial conditions to the above equations are $\rho (0)=\rho_0 (\alpha)$ and $d(0) = d_0 (\alpha) = u_{0x}(\alpha )$ for each $\alpha\in\mathbb{R}$. 
	
Through analysis of system  (\ref{nonlin}) we find the following.		
\begin{theorem}
\label{th1}
For the given 1D Euler Poisson system (\ref{EPMain}) with initial data (\ref{EPMain2}), there is finite time breakdown iff $\exists x \in\mathbb{R}$ such that
\begin{enumerate}
\item
($\nu >2\sqrt{kc}$) Strong damping\\
\begin{align*}
&
\max\{ u_0'(x),\ u_0' (x ) + \lambda_2(c-\rho_0(x))\}<0, \; \text{and} \\
& 
\left| \frac{c\lambda_2 u_0'(x)+k(c-\rho_0(x))}{k\rho_0(x)} \right|^{\lambda_1} \leq \left| \frac{c\lambda_1 u_0'(x) + k(c-\rho_0(x))}{k\rho_0(x)} \right|^{\lambda_2},
\end{align*}
where $\lambda_1=\frac{\nu-\sqrt{\nu^2-4kc}}{2c}$ and $\lambda_2=\frac{\nu+\sqrt{\nu^2-4kc}}{2c}$.

\item
($\nu =2\sqrt{kc}$) Borderline damping
\begin{align*}
& \max\{ u_0'(x),  u_0'(x) + \frac{\nu}{2c} (c-\rho_0(x) \} <0, \; \text{and} \\
& \ln\left(-\frac{ 2cu_0'(x) + \nu (c-\rho_0(x) ) }{\nu\rho_0(x)}\right) \geq \frac{2cu_0'(x)}{2cu_0'(x) + \nu (c-\rho_0(x))}.
\end{align*}

\item
($\nu <2\sqrt{kc}$) Weak damping
\[\left( u_0'(x) + \frac{\nu (c-\rho_0(x) )}{2c}\right)^2\geq \mu^2\left[\frac{\rho_0^2(x)}{c^2}\left( \frac{kce^{\nu t^*}}{\mu^2} - 1 \right)+\frac{(2\rho_0(x) -c)}{c}\right],\]
where $\mu=\sqrt{kc-\nu^2 /4}$ and
\[t^* = \frac{1}{\mu}\left[\beta + tan^{-1}\left( \frac{2\mu u_0'(x)}{\nu u_0'(x) + 2 k(c-\rho_0(x))}\right)\right], \]
\[\beta =\left\{ \begin{array}{lr}
0 & \nu u_0'(x)<min\{ 0,2 k(\rho_0(x) -c)  \}, \\
\pi &  2 k(\rho_0(x) -c)<\nu u_0'(x),\\
2\pi &  0<\nu u_0'(x) <2 k(\rho_0(x) -c).
\end{array}\right.\]

\end{enumerate}
\end{theorem}

This reconfirms the rather remarkable phenomena investigated in \cite{ELT01}, namely that the non-zero background is able to balance the nonlinear convective effects, damping,  and the repulsive forces, to yield a global smooth solution 
if the initial data is within the threshold region.

The above result improves and extends the result stated and proved in \cite{ELT01}. 

In the next theorem we present critical thresholds in an alternative form which can be determined from the phase plane analysis. 
\begin{theorem}
\label{th2}
Consider the 1D Euler Poisson system (\ref{EPMain}) subject to $C^1$ initial data (\ref{EPMain2}). There exists a unique solution 
$\rho,u\in C^1(\mathbb{R}\times (0,\infty))$ iff $\forall x\in\mathbb{R}$, 
\begin{enumerate}
\item
($\nu >2\sqrt{kc}$) Strong damping\\
\[
(\rho_0(x ) ,u_0'(x ) )\in \{ (\rho,d): d >-\rho Q_a(1/\rho ),\, \rho>0  \},
\]
where $Q_a:[0,\infty)\longrightarrow[0,\infty)$ is a continuous function satisfying
\begin{equation}\label{qa}
\frac{dQ_a}{ds} = \nu + \frac{k}{Q_a}(1-cs), \; Q_a(0)=0.
\end{equation}

\item
($\nu =2\sqrt{kc}$) Borderline damping\\
\[(\rho_0(x ) ,u_0'(x ) )\in \{ (\rho,d): d >-\rho Q_b(1/\rho),\, \rho>0  \},\]
where $Q_b:[0,\infty)\longrightarrow[0,\infty)$ is a continuous function satisfying
\begin{equation}\label{qb}
\frac{dQ_b}{ds} = 2\sqrt{kc}+ \frac{k}{Q_b}(1-cs), \; Q_b(0)=0.
\end{equation}
\item 
($\nu <2\sqrt{kc}$) Weak damping\\
\[
(\rho_0(x ) ,u_0'(x ) )\in \{ (\rho,d): -\rho Q_1 (1/\rho) <d <-\rho Q_2(s^* - 1/\rho ),\, \rho\in (1/s^*,\infty)\}, 
\]
where $s^*>0$ is uniquely determined, and $Q_1:[0,s^*]\longrightarrow\mathbb{R}^+\cup \{ 0\}$ is a continuous function satisfying 
\[\frac{dQ_1}{ds} = \nu + \frac{k}{Q_1}(1-cs),\; Q_1(0)=0,\]
and  $Q_2: [0,s^*]\longrightarrow\mathbb{R}^-\cup\{ 0\}$ is another continuous function satisfying 
\[
\frac{dQ_2}{ds} = -\nu + \frac{k}{Q_2}(c(s^*-s)-1),\quad Q_2(0)=0.
\]
\end{enumerate}
\end{theorem}
The details of the proofs of Theorems \ref{th1} and \ref{th2} are carried out in Section \ref{AES} and Section \ref{CTC}, respectively. 

The main  tool in our analysis is a transformation of variables with which we can reduce the non-linear system of equations into a linear system. The solutions to the linear system can then be analyzed or analytically found.  More precisely, we introduce 	
	\begin{subequations} \label{var}
	\begin{align}
	r & = -\frac{d}{\rho},\\
	s & = \frac{1}{\rho},
	\end{align}
	\end{subequations}
so that (\ref{nonlin}) reduces to 
\begin{subequations}
\label{lin}
\begin{align}
r' & = -\nu r - k(1-cs),\\
s' & = -r. 
\end{align}
\end{subequations}
Clearly,  given any initial data, we can find the solution curves $r(t)$ and $s(t)$ which exist for all $t\in \mathbb{R}$.  In order to return to the original unknowns $(\rho, d)$, we need to make sure that $ s $ remains greater than $ 0 $ for all $ t>0 $. In this way, we avoid the finite time breakdown of $ \rho $. Consequently, we find $d$ to be bounded from below since 
\[
d = -r\rho >-\infty\ \forall t>0.
\]
{In all fairness, we should point out that actually $s=\Gamma/\rho_0$, where $\Gamma$ is nothing but the `indicator' function introduced in \cite{ELT01} to denote $\partial_\alpha  x(t; \alpha)$.  The variable $r$ has also appeared as $\beta(t)$ in \cite{ELT01} in the case $\nu=0$ and $c=0$. 
 Here with these two variables combined, we obtain the novel system (\ref{lin}) for the first time. The linearity of (\ref{lin}) and its special structure allow us to derive explicit solutions as shown in Section 3,  which is essentially the same as those preformed  in \cite{ELT01, CCZ} using the flow map techniques. 
However, the geometric structure in terms of phase plane analysis we present here has not been reported in the literature.  
We hope this approach to study the geometric structure of critical threshold curves can be extended to other systems in a  dynamic way.     
}

\section{Proof of Theorem \ref{th1}: Analyzing the explicit solution}\label{AES}
Differentiating (\ref{lin}b) and using (\ref{lin}a), we obtain the following initial value problem (IVP) for $s$,
\begin{subequations}\label{IVP}
\begin{align}
&s'' + \nu s' + kcs = k,\\
& s(0) = \frac{1}{\rho_0}, \; s'(0) = \frac{d_0}{\rho_0}.
\end{align}
\end{subequations}
The type of damping pertains to the type of solutions to this IVP. 

\subsection{Strong damping ($\nu >2\sqrt{kc}$)}
\label{SR}
On solving (\ref{IVP}), we get,
\[s(t) = \frac{1}{c}\left[ 1 + Ae^{-\lambda_1 ct} + Be^{-\lambda_2 ct}  \right],\]
where
\[A = \frac{1}{(\lambda_2 - \lambda_1)\rho_0}\left[ d_0 + \lambda_2(c-\rho_0)\right],\]
\[B = -\frac{1}{(\lambda_2 - \lambda_1)\rho_0}\left[  d_0 + \lambda_1(c-\rho_0)   \right],\]
\[\lambda_1 =\frac{\nu-\sqrt{\nu^2-4kc}}{2c},\quad 
\lambda_2 = \frac{\nu+\sqrt{\nu^2-4kc}}{2c}.
\]
Also note that if either of $ A $ or $ B $ is $ 0 $, then $s(t)>0$ for all $t>0$ is trivially achieved. Since our solution comprises negative exponentials, $ s $ decays to $ 1/c $. Also, such expressions can have at most one extremum. Therefore, on observation we conclude that if $ s' (0) \geq 0 $, then $ s $ remains positive; that is, $ d_0 \geq 0 $ ensures global existence. Next, we differentiate the expression for $ s $,
\[s'(t) =  -\left[  A\lambda_1 e^{-\lambda_1 ct} + B\lambda_2 e^{-\lambda_2 ct}  \right],\]
and equate it to $ 0 $ to obtain an expression for time $ t=t^* $ at which the extrema occurs,
\[ e^{(\lambda_2 - \lambda_1)ct^*} = -\frac{B\lambda_2}{A\lambda_1}. \]
Furthermore, we differentiate $ s' $ again to obtain,
\[ s'' (t) =  c\left[  A\lambda_1^2 e^{-\lambda_1 ct} + B\lambda_2^2 e^{-\lambda_2 ct}  \right], \]
and write it in the following way,
\[  = Ace^{-\lambda_1 ct}\lambda_1\left[ \lambda_1 + \frac{B}{A}\frac{\lambda_2^2}{\lambda_1}e^{-(\lambda_2 - \lambda_1)ct} \right], \]
and substitute the expression for $ e^{-(\lambda_2 - \lambda_1)ct^*} $ to obtain,
\[ s''(t^*) = -(\lambda_2 - \lambda_1)Ace^{-\lambda_1 ct^*}\lambda_1. \]
Hence, if $ A>0 $ then the extremum, if it exists, is a maximum. Furthermore, if $A<0$, we conclude from the expression of $t^*$ that $B>0$ becomes a necessary condition for $t^*>0$ to exist. Therefore, $A>0$ ensures global solution anyways. To obtain the critical threshold curve, we apply the condition that if there exists $ t=t^* >0 $, where the minimum of $ s $ is achieved, then $ s(t^*)>0 $ which, from (\ref{IVP}a), is equivalent to the condition that $s''(t^*)<k$. Therefore, we get
\[
-(\lambda_2 - \lambda_1)Ace^{-\lambda_1 ct^*}\lambda_1 <k.
\]
After substituting the expression for $t^*$, the above equation can be rewritten as,
\[-(\lambda_2 - \lambda_1)Ac\lambda_1 < k\left( -\frac{B\lambda_2}{A\lambda_1}  \right)^{\frac{\lambda_1}{(\lambda_2 - \lambda_1)}}.\]
Noting that $A<0$ and $B>0$, we get
\[c(\lambda_2 - \lambda_1 )(-A\lambda_1)^{\frac{\lambda_2}{\lambda_2-\lambda_1}}<
k(B\lambda_2)^{\frac{\lambda_1}{\lambda_2 -\lambda_1}}.
\]
On substituting the expressions for $A$ and $B$ and using that the difference of the exponents on either side of the inequality is one, 
we obtain 
\begin{equation}
\label{ineq}
\left[ -\frac{c\lambda_1 d_0+k(c-\rho_0)}{k\rho_0} \right]^{\lambda_2} < \left[ -\frac{c\lambda_2 d_0 + k(c-\rho_0)}{k\rho_0} \right]^{\lambda_1}.
\end{equation}
From the discussion above, we also need $d_0<0$ and  $d_0+\lambda_2 (c-\rho_0)<0$ for uniquely determining the curve as a critical threshold.  These together lead to the conclusion in (1) of Theorem \ref{th1}. 
\subsection{Weak damping ($\nu <2\sqrt{kc}$)}
\label{WR}
On solving (\ref{IVP}), we get,
\[s (t) = \frac{1}{c} +  \frac{e^{-\nu\frac{t}{2}}}{c}\left[ \left(\frac{c-\rho_0}{\rho_0}\right)\cos\mu t + \frac{1}{\mu}\left( \frac{cd_0}{\rho_0} + \frac{\nu (c-\rho_0)}{2\rho_0}\right)\sin\mu t  \right],\]
where $\mu = \sqrt{kc-0.25\nu^2}$. Once again, the idea is that if $s= (1/\rho )$ becomes $0$ at some time $t = t_c$, means $\lim_{t\rightarrow t_c^-}\rho(t) = \infty$, then there is breakdown of the solution. Since, the solutions are decaying with time, we need the first local minimum of $s$ to be greater than $0$. In view of this, we find all the times for which
\[s'(t) = 0,\quad s'' (t)>0.\]
In view of this, we calculate
\[
s' (t) = e^{-\nu\frac{t}{2}}\left[ \left( \frac{d_0}{\rho_0} \right)\cos (\mu t)     -\left( \frac{\nu d_0 + 2 k(c-\rho_0 )}{2 \mu\rho_0}  \right)\sin (\mu t)\right].
\]
So $s'(t)=0$ if 
\begin{equation}\label{tan}
\tan(\mu t) = \frac{2\mu d_0}{\nu d_0 + 2 k(c-\rho_0)}.
\end{equation}
We further perform the second derivative test,
\[s'' (t) = \frac{ce^{-\nu\frac{t}{2}}}{\rho_0}\cos (\mu t)\left[ \tan (\mu t)\left(  \frac{-4\mu^2 d_0 + \nu^2 d_0 + 2k\nu (c-\rho_0 )}{4\mu} \right) - \frac{\nu d_0 + 2k(c-\rho_0 )}{2} -\nu\frac{d_0}{2} \right],
\]
which in virtue of (\ref{tan}) gives 
\[
 s'' (t)= -\frac{ce^{-\frac{\nu t}{2}}\cos(\mu t)}{\rho_0(\nu d_0 + 2k(c-\rho_0 ))} \left[ 2 \mu^2 d_0^2 + \frac{\left( \nu d_0 + 2 k(c-\rho_0 ) \right)^2}{2}  
 \right]. 
 \] 
 The condition for a minimum to occur is $ sgn\{ \cos (\mu t )\} = -sgn\{ \nu d_0 + 2 k(c-\rho_0 )\}$. This gives us a sequence of minima at different times. 
 At the first time $t^*$, we have  
\[\mu t^*  = \beta + tan^{-1}\left( \frac{2\mu d_0}{\nu d_0 + 2 k(c-\rho_0)}\right), \]
\[\beta =\left\{ \begin{array}{lr}
0 & \nu d_0<min\{ 0,2 k(\rho_0 -c)  \}, \\
\pi &  2 k(\rho_0 -c)<\nu d_0,\\
2\pi &  0<\nu d_0 <2 k(\rho_0 -c).
\end{array}\right.\]
Note that for 
\[f(t) = c_0+ e^{-\gamma t}\left[ c_1 \cos (\theta t) + c_2\sin (\theta t)\right],\quad \gamma,\theta >0,\]
the value of the function $f$ at a local minima $\tau$ is 
\[ c_0 - e^{-\gamma \tau }\frac{\theta\sqrt{c_1^2 + c_2^2}}{\sqrt{\theta^2 + \gamma^2}}.
\]
On comparing $s$ with $f$ above we have,
\[\gamma = \frac{\nu}{2},\quad \theta = \mu,\; c_0=1/c, \quad c_1=\frac{c-\rho_0}{\rho_0},\quad c_2=\frac{1}{\mu}\left( \frac{cd_0}{\rho_0}+\frac{\nu (c-\rho_0 )}{2\rho_0}\right). \] 
Applying the above formula on $s(t)$, we have for $s(t^*)>0$ to hold, 
\[1-e^{-\frac{\nu t^*}{2}}\mu\frac{\sqrt{\left( \frac{c-\rho_0}{\rho_0} \right)^2 + \frac{1}{\mu^2}\left( \frac{cd_0}{\rho_0} + \frac{\nu (c-\rho_0) }{2\rho_0} \right)^2}}{\sqrt{kc}} >0. \]
This is equivalent to the following 
\begin{equation}
\label{wrineq}
\left( d_0 + \frac{\nu (c-\rho_0 )}{2c}\right)^2<\mu^2\left[\frac{\rho_0^2}{c^2}\left( \frac{kce^{\nu t^*}}{\mu^2} - 1 \right)+\frac{(2\rho_0 -c)}{c}\right].
\end{equation}
This proves (2) in Theorem \ref{th1}. 

\subsection{Borderline damping ($\nu = 2\sqrt{kc}$)}
On solving (\ref{IVP}), we get
\[s(t) = \frac{1}{c} + \left[ D+\left( \frac{d_0}{\rho_0}+\frac{D\nu}{2}\right) t \right]e^{-\frac{\nu t}{2}},\]
where $D = \frac{1}{\rho_0}-\frac{1}{c}.$  Setting $s'(t^*)=0$ to obtain extremum, hence from 
$$
s'(t) = \left[ \frac{d_0}{\rho_0}-\left( \frac{d_0}{2\rho_0} + \frac{D\nu}{4} \right)\nu t \right]e^{-\frac{\nu t}{2}}, 
$$
we get
\[
t^* = \frac{4d_0}{\nu\left( 2d_0 + \nu D\rho_0 \right)},
\] 
which is positive if $d_0<0$ and 
$$
2d_0 + \nu D\rho_0<0. 
$$
The latter ensures that $s''(t^* )>0$ since 
\[s'' (t) = -\nu\left[ \frac{d_0}{\rho_0} + \frac{D\nu}{4} - \left( \frac{d_0}{\rho_0} + \frac{D\nu}{2} \right)\frac{\nu t}{4} \right]e^{-\frac{\nu t}{2}}.
\]
The above $t^*$ when inserted into $s(t^*)>0$ gives us,
\[
-c\left( D + \frac{2d_0}{\nu\rho_0} \right) < e^{\frac{\nu t^*}{2}},
\]
which  is equivalent to the following 
\begin{equation}
\label{brineq}
\ln\left(-\frac{\left( 2cd_0 + \nu (c-\rho_0 ) \right) }{\nu\rho_0}\right) < \frac{2cd_0}{2cd_0 + \nu (c-\rho_0)}
\end{equation}
provided $d_0<0$ and $2cd_0 + \nu (c-\rho_0)<0$. This proves (3), hence completes the proof of Theorem \ref{th1}.

\section{Proof of Theorem \ref{th2}: Critical threshold curve}
\label{CTC}
We will look into the geometrical interpretation of the critical threshold curve for the 3 cases in this section.  First, we note  from (\ref{lin}) that $(0,1/c)$ is the only critical point in phase plane, and the vector field for the system (\ref{lin}) is shown in Figure \ref{fig1}. 

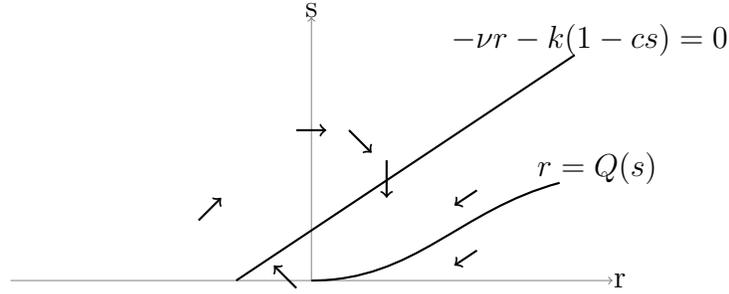
\begin{figure}[H]
\centering
\begin{tikzpicture}
  
  \draw [help lines, ->] (-4,0) --(4,0);
  \draw [help lines, ->] (0,0) -- (0,3.5);
  \node at (4.1,0) {r};
  \node at (0,3.6) {s};
  \draw[ thick] (-1,0) -- (3.5,3);
  \node at (3.7,3.2) {$-\nu r - k(1-cs)=0$};
  \draw[thick] (0,0) to [out=0,in=195] (3.3,1.3);
  \node at (3.8,1.5) {$r=Q(s)$};
  \draw[thick,->] (2.2,0.4) -- (1.9,0.2);
  \draw[thick,->] (-0.2,-0.1) -- (-0.5,0.2);  
  \draw[thick,->] (-1.5,0.8) -- (-1.2,1.1); 
  \draw[thick,->] (-0.2,2) -- (0.2,2); 
  \draw[thick,->] (0.5,2) -- (0.8,1.7); 
  \draw[thick,->] (1,1.6) -- (1,1.1); 
  \draw[thick,->] (2.2,1.2) -- (1.9,1);  
   
\end{tikzpicture}
\caption{Vector field for (\ref{lin}).}
\label{fig1}
\end{figure}
\vskip 0.2in

The key point is that in the $r-s$ plane, $s=0$ corresponds to $\rho = \infty$ by (\ref{var}b). 
We need to identify an invariant region $\Sigma$ in phase plane so that $s(t)>0$ for all $t>0$ if $(r_0, s_0)\in \Sigma$. 
Its boundary when transformed onto the $\rho -d$ plane through (\ref{var}) would give us the critical threshold curve. By observation, a trajectory curve starting at the origin and moving backwards in time would give us $\partial \Sigma$, the boundary of $\Sigma$.  

We proceed to discuss each case as stated in Theorem \ref{th2}. 

\subsection{Strong damping ($\nu >2\sqrt{kc}$)}
\label{SRCTC}
We rewrite (\ref{lin}) as the following form 
\begin{align}\label{coe}
\frac{d}{dt}
\left(
\begin{array}{c}
 r\\
 s-1/c
 \end{array}
\right)=  \left(\begin{array}{cc}
 -\nu & ck \\
 -1 & 0
  \end{array}
\right) \cdot \left(\begin{array}{c}
 r\\
 s-1/c
  \end{array}
\right).
\end{align}
The coefficient matrix on the right hand side has eigenvalues  
 $$
 \lambda_{\pm}=(-\nu\pm\sqrt{\nu^2-4kc})/2
 $$ 
with the corresponding eigenvectors 
$$
v_-= (-\lambda_-, 1)^\top, \quad v_+=(-\lambda_+, 1)^\top.
$$
Under the strong damping condition, the critical point $(0,1/c)$ is an asymptotically stable node. In order to determine $\Sigma$, we know that the boundary curve $(r(t), s(t))$ satisfy the following 
$$
r'=-\nu r +k(cs-1), \quad s'=-r, \quad t<\tau
$$
with $(r, s)=(0, 0)$ at a time $t=\tau$.  When the time parameter is eliminated, we obtain 
$$
\frac{dr}{ds} = \nu + \frac{k}{r}(1-cs), \quad r(0)=0. 
$$
Let such a trajectory be denoted by $r=Q_a(s)$, we have   
$$
\Sigma=\{(r, s), \quad s>0, r<Q_a(s)\},
$$
where $Q_a(s)$ is as defined in (\ref{qa}). We now show such set is well-defined by looking at the asymptotic behavior of $Q_a(s)$. Since both the eigenvalues are real and negative, then 
$$
\lim_{t\to -\infty} \frac{r(t)}{s(t)}=-\lambda_-, 
$$
the slope of the eigenvector $v_-$, for any trajectories. Hence 
 \[
 \lim_{s\rightarrow\infty}\frac{Q_a(s)}{s} = \frac{\nu + \sqrt{\nu^2 - 4kc}}{2}.
 \]
One can show that $r=-\lambda_- (s-1/c)$ is a trajectory, so the curve $r=Q_a(s)$ always remains below it. As a result,  $r=Q_a(s)$ also remains  below the line $\nu r +k(1-cs)=0$. 
Hence, we have  $\frac{dQ_a}{ds}>0$ for $s\in(0,\infty )$,  $\Sigma$ is thus well-defined. 
We can conclude that $s(t)>0\ \forall t>0$ if and only if $(r_0,s_0)\in \Sigma$. 

We also need to ascertain the behavior of $\frac{Q_a(s)}{s}$ as $s$ goes to zero to know the behavior of $d$ versus $\rho$ on the $\rho-d$ plane as $\rho \to \infty$.  First we know that
\[\lim_{s\rightarrow 0^+}\frac{Q_a(s)}{s} = \lim_{s\rightarrow 0^+}Q_a'(s) = \lim_{s\rightarrow 0^+} \nu + \frac{k}{Q_a(s)} = +\infty ,\]
which from (\ref{var}) implies that $d\rightarrow -\infty$ as $\rho\rightarrow\infty$.  Transforming the threshold curve back onto the $\rho -d$ plane through (\ref{var}), there is global solution if and only if
 \[(\rho_0 ,d_0)\in \{ (\rho ,d):d>-\rho Q_a(1/\rho),\, \rho\in (0,\infty ) \}.\]
\begin{remark}
\label{remnonlin1}
We could also evaluate $\lim_{s\rightarrow\infty}\frac{Q_a(s)}{s}$ using (\ref{nonlin}). Since $r=Q_a(s)$ is a trajectory, using (\ref{var}), the above limit is the value of $-d$ as $\rho\rightarrow 0$. Since $\rho=0$ is a solution to (\ref{nonlin}a), we thus have 
\[
d' = -(d^2 + \nu d + kc),
\]
which gives 
\[d' = -(d-\lambda_+)(d-\lambda_-).
\]
For this Ricatti equation, $d$ breaks down in finite time if and only if initial data $d_0 <\lambda_-$. 
Therefore, 
\[\lim_{s\rightarrow\infty}\frac{Q_a(s)}{s} = -\lambda_- = \frac{\nu + \sqrt{\nu^2 - 4kc}}{2}. \]
\end{remark} 

\subsection{Borderline damping ($\nu = 2\sqrt{kc}$)}
\label{BRCTC}
 Note that $(0,1/c)$ is an asymptotically stable improper node with eigenvalue $\lambda=-\nu/2$ and the corresponding eigenvector $v=[-\lambda\quad 1]^T$.   Similar to the strong damping case we can identify the invariant region
$$
 \Sigma=\{(r, s), \quad s>0, r<Q_b(s)\},
$$ 
where $Q_b$ is a monotone function, satisfying 
\[
\lim_{s\rightarrow\infty}\frac{Q_b(s)}{s}= \frac{\nu}{2},
\]
and it can be determined by the ODE (\ref{qb}). We also have 
\[\lim_{s\rightarrow 0^+}\frac{Q_b(s)}{s}  = +\infty.
\]
These enable us to 
conclude that  $s(t)>0\ \forall t>0$ if and only if $(r_0,s_0)\in \Sigma$. 
Transforming the invariant region back onto the $\rho$-$d$ plane through (\ref{var}), there is global solution if and only if
\[(\rho_0 ,d_0)\in \{ (\rho ,d):d>-\rho Q_b(1/\rho),\, \rho\in (0,\infty ) \}.\]
We point out that on the $\rho$-$d$ plane, the critical threshold curve starts at $(0,-\nu /2)$ and monotonically goes to negative infinity as $\rho$ goes to infinity.\\

\subsection{Weak damping ($\nu <2\sqrt{kc}$)}
\label{WRCTC} In such case, the coefficient matrix on the right hand side of (\ref{coe})  has eigenvalues  
 $$
 \lambda_{\pm}=-\frac{\nu}{2}  \pm  \frac{i}{2}\sqrt{4kc-\nu^2}.
 $$ 
Hence  $(0,1/c)$ is an asymptotically stable spiral point for system (\ref{lin}). Therefore, trajectories spiral into the critical point as time increases. Consequently, by vector field diagram, a trajectory beginning at the origin and proceeding backwards in time (into the first quadrant on $r$-$s$ plane) would spiral outwards and hit the $s=0$ line again in the second quadrant. This segment of the trajectory is the threshold curve on the $r$-$s$ plane. This curve partitions the upper half plane into two sections. The closed region formed by this curve and $s=0$ would then be an invariant region $\Sigma$ since any trajectory with initial data $(r_0,s_0)$ in this region would spiral inwards clockwise without touching the $s=0$ line. And any trajectory with initial data outside this region remains outside this region. 

We proceed to describe $\Sigma$ using two functions.
\begin{figure}[H]
\centering
\begin{tikzpicture}
  
  \draw [help lines, ->] (-4,0) --(4,0);
  \draw [help lines, ->] (0,0) -- (0,3.5);
  \node at (4.1,0) {r};
  \node at (0,3.6) {s};
  \draw[ thick] (-1,0) -- (3.5,3);
  \node at (3.7,3.2) {$-\nu r - k(1-cs)=0$};
  \node at (0.2,2.8) {$s^*$};
  \draw[thick] (0,0) to [out=0,in=270] (1,1) to [out=90,in=0] (0,2.7) to [out=180,in=60] (-3,0);
  \node at (-2,2.5) {$r=Q(s)$};
  \draw[thick,->] (2.2,0.4) -- (1.9,0.2);
  \draw[thick,->] (-0.2,-0.1) -- (-0.5,0.2);  
  \draw[thick,->] (-1.5,0.8) -- (-1.2,1.1); 
  \draw[thick,->] (-0.2,2) -- (0.2,2); 
  \draw[thick,->] (0.5,2) -- (0.8,1.7); 
  \draw[thick,->] (1,1.6) -- (1,1.1); 
  \draw[thick,->] (2.2,1.2) -- (1.9,1);  
   
\end{tikzpicture}
\caption{The curve along with vector field for (\ref{lin}).}
\label{fig2}
\end{figure}
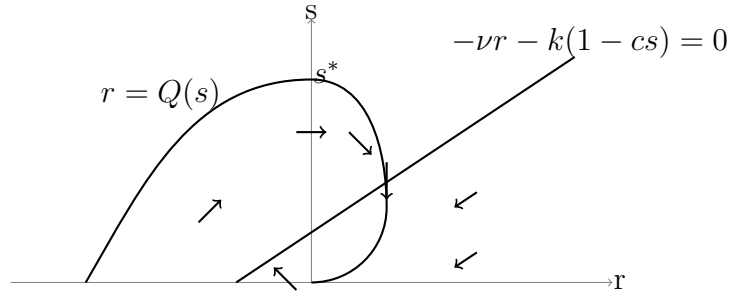
Because of the outward spiraling nature of the trajectory curve starting from $(0, 0)$, then there exists a unique  $s^*>1/c$ so that we have 
$(r, s)=(0, s^*)$ at some $t$.  The first segment of such curve can be defined via a continuous function, 
\[Q_1 :[0,s^* ]\longrightarrow \mathbb{R}^+\cup \{ 0\},\]
satisfying for $0<s<s^*$, 
\[\frac{dQ_1}{ds} = \nu + \frac{k}{Q_1}(1-cs), \quad Q_1(0)=0.\]
Define a continuous function $Q_2 : [0,s^* ]\longrightarrow\mathbb{R}^-\cup \{ 0\}$ as follows. 
\[
\frac{dQ_2}{d\tau} = -\nu + \frac{k}{Q_2(\tau)}(cs^* -1 -c\tau), \quad Q_2(0)=0.
\]
Then $r=Q_2(s^*-s)$ gives the left segment of the said trajectory curve in the $r-s$ plane. That is the invariant region can be defined by 
\[
\Sigma=\{ (r,s): Q_2 (s^*-s) < r < Q_1 (s), s\in (0,s^*) \}.
\]
In order to transform $\Sigma$ back to the $\rho$-$d$ plane, we evaluate the appropriate limits to ascertain the behavior of the threshold curve as $\rho \to \infty$ 
As done before, we have 
$$
\lim_{s\rightarrow 0^+}\frac{Q_1(s)}{s} = \lim_{s\rightarrow 0^+}\frac{dQ_1}{ds}(s) = +\infty.
$$
Note that $Q_2(s^*)<0$ due to the outward spiraling nature of the threshold trajectory,  hence  
\[
\lim_{s \rightarrow 0}\frac{Q_2(s^* -s)}{s}=-\infty.
\]
We can now transform back onto the $\rho -d$ plane through (\ref{var}) to conclude that  there is global solution if and only if
\[(d_0,\rho_0 )\in \{ (d,\rho): -\rho Q_1 (1/\rho) <d <-\rho Q_2(s^* - 1/\rho ),\ \rho\in (1/s^* ,\infty ) 
 \}.
\]
Note that the shape of the critical threshold curve is similar to that of a parabola opening towards positive $\rho$ axis and vertex at $(1/s^* ,0)$. This completes the proof to Theorem \ref{th2}.

\begin{remark} Using (\ref{wrineq}), we can find $s^*$ explicitly. Note that the left-most point of the threshold curve is $(\rho^*, d^*)=(1/s^*, 0)$ with $\rho^*<c$; for which $\beta =\pi$, setting $d_0=d^*=0$ in  (\ref{wrineq}) and using $\mu^2 = kc - \nu^2 /4$, we find that 
\[
\rho^*= \frac{c}{e^{\frac{\nu\pi}{2\mu}}+1}.\]
Hence, $s^* = \frac{1}{\rho^*}=\frac{e^{\frac{\nu\pi}{2\mu}}+1}{c}$.
\end{remark}

\section{Application to an aggregation system} In this section, we illustrate that our results can be applied to system (\ref{cc}), subject to $C^1$ initial data (\ref{ccid}). 

For initial data $(\rho_0, u_0)$ defined on the entire $\mathbb{R}$, we shall make the following assumptions concerning their behavior at far fields.  There exists $\delta >0$ such that 
\begin{align}\label{ru}
u_{0x}(x)  \in C_b^0(\mathbb{R}), \quad \langle x \rangle^{2+\delta}\rho_0(x) \in C_b^0(\mathbb{R}), 
\end{align}
where $\langle x \rangle:=\sqrt{1+x^2}$, and 
$C_b^0(\mathbb{R})$ denotes the set of bounded continuous functions on $\mathbb{R}$. Under (\ref{ru}) we have 
\begin{align}\label{m01}
\int_{\mathbb{R}}\rho_0(x)\, dx<\infty, \quad \int_{\mathbb{R}} \rho_0(x) |u_0(x)|\, dx<\infty,\quad \int_{\mathbb{R}}|x|\rho_0 (x)\, dx <\infty.
 \end{align}
 We state the local wellposedness  in the following.
\begin{theorem} \label{l2} ($\textbf{Local existence}$) If $\rho_0,\, u_0\in C^1(\mathbb{R})$, and (\ref{ru}) is satisfied,  
then there exists $T>0$,  depending upon the initial data so that (\ref{cc}), (\ref{ccid}) has a unique solution $(\rho,u)\in C^1(\mathbb{R}\times [0,T))$, that  
for $t\in [0, T)$ and $i=1, 2$ 
\begin{align}\label{rut}
 \rho(x,t)|u|^i (x,t)\to 0\ as\ |x|\to\infty \quad \text{and} \quad   \rho(x,t)\leq\frac{C}{|x|^{2+\delta}}.
\end{align}
Moreover,  if the maximum life span $T^* < \infty$, then
$$
\lim_{t \rightarrow T^*-} \partial_x u(t, x^*) =-\infty
$$
for some $x^* \in \mathbb{R}$.
\end{theorem}

In other words, we prove local existence and uniqueness for solutions in a more restricted function space. We would like to point out that the decaying assumptions at far fields make sense physically and are reasonable assumptions as we want the particle density, momentum, and energy to vanish at $\pm \infty$. 

Using \eqref{m01}, set 
\[
\int_\mathbb{R}\rho_0(x)\, dx =: M_0,\quad \int_\mathbb{R}\rho_0(x)u_0(x)\, dx =: M_1.
\]
For $C^1$ solutions satisfying (\ref{rut}) we first derive a local PDE in terms of $(u, E, d, \rho)$ with $E := \partial W\star\rho$ and $d=u_x$, and reformulate it into a closed ODE system.  
First, we integrate (\ref{cc}a) to get
\[\int_\mathbb{R}\rho(x,t)\, dx = M_0,\]
since $\rho u\to 0$ as $|x|\to\infty$. Hence we have  
\begin{align}\label{E}
E(x,t) & = \int_\mathbb{R} [-sgn(x-y)+x-y]\rho(y,t)\, dy \\ \notag 
& =  (x+1) M_0 - \int_\mathbb{R} y \rho(y,t)dy -2\int^x_{-\infty} \rho(y, t)dy.
\end{align}
Here we see that $E$ is well defined due to (\ref{rut}).  (\ref{cc}) can be used to obtain $(\rho u)_t + (\rho u^2)_x + \rho u = -E\rho$. Upon integrating,
\[\frac{d}{dt}\int_\mathbb{R} \rho u + \int_\mathbb{R} \rho u = 0. \]
Here, we used the decay of $\rho u^2$ and symmetry of $E$ so that $\int_\mathbb{R} E\rho dx=0$. 
Therefore,
\[
\int_\mathbb{R} \rho(y,t)u(y,t)\, dy = M_1e^{-t}.
\]
Differentiation of (\ref{E}) with respect to $t$ using (\ref{cc}a) yields 
\[E_t = 2\rho u- \int_\mathbb{R} \rho(y,t)u(y,t)\, dy= 
2\rho u-M_1e^{-t}.
\]
Also, we get $E_x =M_0 -2\rho$, so that 
$$
E_t +uE_x=M_0 u-M_1e^{-t}. 
$$
Together with the equation for $d=u_x$ we obtain an augmented system 
\begin{subequations}\label{closedPDE}
\begin{align} 
& \rho_t +u\rho_x + \rho d=0,\\
& u_t +uu_x + u = -E,\\
& d_t +u d_x + d^2 + d = -E_x=2\rho-M_0,\\
& E_t+uE_x = -M_1 e^{-t} + M_0 u.
\end{align}
\end{subequations}
From this system we derive the characteristic system, based on which we further construct the local-in-time solution. Finally, we show  such constructed solution indeed satisfies (\ref{rut}). Further details will be deferred to Appendix. 

Since both equations for $u$ and $E$ are linear, and decoupled from the equations for $\rho$ and $d$. It suffices to consider  
 the following system of equations \begin{subequations} to find the critical threshold:  
\label{nonlocnonlin}
\begin{equation}
\label{nonlocalnonlin1}
\rho' + \rho d=0,
\end{equation}
\begin{equation}
\label{nonlocalnonlin2}
d' + d^2 + d = 2\left(\rho - \frac{M_0}{2}\right),
\end{equation}
\end{subequations}
and 
	\[
	\{\}' = \frac{\partial}{\partial t} + u\frac{\partial}{\partial x},
	\]
	 denotes the differentiation along the particle path,
	$$
	\Gamma=\{(x, t)|\;  x'(t)=u(x(t), t), x(0)=\alpha \in \mathbb{R}\}.
	$$
This is a particular case of system  (\ref{nonlin}), with 
\[\nu = 1,\qquad k=2,\qquad c=\frac{M_0}{2}.\]
Consequently, we have the following theorems.
\begin{theorem}
\label{thappl1}
For the given 1D pressureless damped Euler system of equations (\ref{cc}) subject to initial data (\ref{ccid}), there is finite time breakdown iff $\exists x \in\mathbb{R}$ such that
\begin{enumerate}
\item
(Subcritical mass $M_0 <1/4$) 
\begin{align*}
& \max\{ u_0'(x),\ u_0' (x ) + \lambda_2(0.5M_0-\rho_0(x))\}<0 \;\text{and} \;  \\
& \left| \frac{0.5M_0\lambda_2 u_0'(x)+M_0-2\rho_0(x)}{2\rho_0(x)} \right|^{\lambda_1} \leq \left| \frac{0.5M_0\lambda_1 u_0'(x) + M_0-2\rho_0(x)}{2\rho_0(x)} \right|^{\lambda_2},
\end{align*}
where $\lambda_1=\frac{1-\sqrt{1-4M_0}}{M_0}$ and $\lambda_2=\frac{1+\sqrt{1-4M_0}}{M_0}$.

\item
(Critical mass  $M_0=1/4$) 
\begin{align*}
& \max\{ u_0'(x),  M_0u_0'(x) + 0.5M_0-\rho_0(x)\} <0 \; \text{and} \;  \\
& \ln\left(-\frac{ M_0u_0'(x) +  0.5M_0-\rho_0(x)  }{\rho_0(x)}\right) \geq \frac{M_0u_0'(x)}{M_0u_0'(x) + 0.5M_0-\rho_0(x)}.
\end{align*}

\item
(Supercritical mass  $M_0 >1/4$)
\[\left( u_0'(x) + \frac{0.5M_0-\rho_0(x)}{M_0}\right)^2\geq \left(M_0-\frac{1}{4}\right)\left[\frac{4\rho_0^2(x)}{M_0^2}\left( \frac{M_0e^{t^*}}{M_0-\frac{1}{4}} - 1 \right)+\frac{4\rho_0(x) -M_0}{M_0}\right],\]
where
\[t^* = \frac{1}{\sqrt{\left(M_0-\frac{1}{4}\right)}}\left[\beta + tan^{-1}\left( \frac{2 u_0'(x)\sqrt{\left(M_0-\frac{1}{4}\right)}}{ u_0'(x) + 2M_0-4\rho_0(x)}\right)\right], \]
\[\beta =\left\{ \begin{array}{lr}
0 &  u_0'(x)<min\{ 0,4(\rho_0(x) -0.5M_0)  \}, \\
\pi &  4(\rho_0(x) -0.5M_0)< u_0'(x),\\
2\pi &  0< u_0'(x) <4(\rho_0(x) -0.5M_0).
\end{array}\right.\]

\end{enumerate}
\end{theorem}

\begin{theorem}
\label{thappl2}Consider the given 1D pressureless damped Euler system of equations (\ref{cc}) subject to initial data (\ref{ccid}). There exists a unique global solution $\rho,u\in C^1(\mathbb{R}\times (0,\infty))$ iff $\forall x\in\mathbb{R}$,
\begin{enumerate}
\item
(Subcritical mass $ M_0<1/4$)
\[
(\rho_0(x ) ,u_0'(x ) )\in \{ (\rho,d): d >-\rho R_a(1/\rho ),\, \rho>0  \},
\]
where $R_a:[0,\infty)\longrightarrow[0,\infty)$ is a continuous function satisfying
\[\frac{dR_a}{ds} = 1 + \frac{1}{R_a}(2-M_0 s), \; R_a(0)=0.\]

\item
(Critical mass $M_0=1/4$) 
\[(\rho_0(x ) ,u_0'(x) )\in \{ (\rho,d): d >-\rho R_b(1/\rho),\, \rho>0  \},\]
where $R_b:[0,\infty)\longrightarrow[0,\infty)$ is a continuous function satisfying
\[\frac{dR_b}{ds} = 1 + \frac{1}{R_b}(2-s/4), \; R_b(0)=0.\]
\item
(Supercritical mass $M_0>1/4$) 
\[
(\rho_0(x ) ,u_0'(x ) )\in \{ (\rho,d): -\rho R_1 (1/\rho) <d <-\rho R_2(\gamma - 1/\rho ),\, \rho\in (1/\gamma,\infty)\}, 
\]
where 
$$
\gamma=\frac{2}{M_0} \left( 1+
e^{\frac{\pi}{\sqrt{4M_0-1}}}
\right),
$$
and $R_1:[0,\gamma]\longrightarrow\mathbb{R}^+\cup \{ 0\}$ is a continuous function satisfying 
\[\frac{dR_1}{ds} = 1 + \frac{1}{R_1}(2-M_0 s),\; R_1(0)=0,\]
and  $R_2: [0,\gamma]\longrightarrow\mathbb{R}^-\cup\{ 0\}$ is another continuous function satisfying 
\[
\frac{dR_2}{ds} = -1 + \frac{1}{R_2}(M_0(\gamma-s)-2),\quad R_2(0)=0.
\]
\end{enumerate}
\end{theorem}

\appendix

\section{Proof of Theorem \ref{l2}} 
Denote  $' = \partial/\partial t + u\partial/\partial x$ as the derivative along the particle path,
\begin{equation}\label{flowpath}
\Gamma = \left\{ (x,t):\dot{x} = u(x,t),\ x(0)=\alpha,\, \alpha\in\mathbb{R}\right\},
\end{equation}
and define 
\begin{align*}
& f(\alpha,t):=\rho(x(\alpha,t),t),\\
& Q(\alpha,t):=E(x(\alpha,t),t),\\
& d(\alpha,t):=u_x(x(\alpha,t),t),\\
& z(\alpha,t):=u(x(\alpha,t),t), 
\end{align*}
so to obtain the following closed ODE system 
\begin{subequations}\label{closed}
\begin{align} 
& f' + fd=0,\\
& z' + z = -Q,\\
& d' + d^2 + d = 2f-M_0,\\
& Q' = -M_1 e^{-t} + M_0 z,
\end{align}
\end{subequations}
subject to initial data 
\begin{subequations}\label{closedid}
\begin{align} 
& f(\alpha,0) = \rho_0(\alpha),\\
& z(\alpha,0) = u_0(\alpha),\\
& d(\alpha,0) = u_{0x}(\alpha),\\
& Q(\alpha,0) = E_0(\alpha),
\end{align}
\end{subequations}
where $E_0(\alpha)$ is defined by 
\begin{align} \label{e0}
E_0(\alpha) = \int_\mathbb{R}[-sgn(\alpha - \beta)+(\alpha-\beta)]\rho_0(\beta)\, d\beta 
\end{align}
which is well defined since $|E_0(\alpha)-M_0\alpha| \leq M_0 +\int_{\mathbb{R}} |x|\rho_0(x)dx$.  

 This ODE problem, for each fixed $\alpha_0\in\mathbb{R}$, 
 admits a unique local $C^1$ solution $(f,z,d,Q)(\alpha, t)$ in a neighborhood of $(\alpha_0,0)$.
 Our aim is to find a $T>0$, such that $ f(\alpha,t),z(\alpha,t) $ are in $C^1(\mathbb{R}\times [0,T])$. 
 
 We do this by ensuring that the deformation of the path has a strictly positive uniform lower bound.  In the neighborhood of any $(\alpha_0,0)$, using (\ref{closed}b), (\ref{closed}d), (\ref{closedid}b) and (\ref{closedid}d), we see that $a:=z_\alpha(\alpha,t)$ solves 
\[a'' + a' + M_0a = 0,\]
\[a(0) = d_0(\alpha),\quad a'(0) = -d_0(\alpha)+2\rho_0(\alpha)-M_0.
\]
Since $d_0=u_{0x},\rho_0$ are bounded uniformly in terms of $\alpha$, we find that for any $ t>0 $, $|a(\alpha,t)|\leq C$ uniformly in $\alpha$. Hence, choosing $T<1/(2C)$, we ensure 
\[
\frac{\partial x}{\partial\alpha} = 1 + \int_0^t z_\alpha (\alpha,s)\, ds>\frac{1}{2}\quad \forall\alpha\in\mathbb{R},\ t\in [0,T].
\]

Eventually, by the inverse function theorem, for each $x\in \mathbb{R}$ and $t\in (0, T]$,  we can  uniquely solve the equation 
$$
x=x(\alpha, t), \quad \alpha=\alpha(x, t),
$$
where the mapping $x \to \alpha, t$ is $C^2$. Finally let us define 
\begin{align*}
& \rho(x, t):=f(\alpha(x, t),t),\\
& E(x,t)=Q(\alpha(x, t),t),\\
& u(x,t)=z(\alpha(x, t),t),\\
&p(x, t)=d((\alpha(x, t),t)
\end{align*}
for each $x\in \mathbb{R}$ and $t\in (0, T]$. It remains to show $(\rho, u)$ is indeed a solution to \eqref{cc}. One can very that 
$
f'=\rho_t +u \rho_x,
$
so that 
$$
\rho_t +u \rho_x+ \rho p=0.
$$
We still need to show $p=u_x$, that is to show $d=z_\alpha \cdot \frac{\partial \alpha}{\partial x}$ along the particle path $\Gamma$.  In view of this, fix $\alpha$ and set
\[\Theta(t) := z_\alpha - \frac{\partial x}{\partial \alpha}d.\]
Upon differentiating,
\begin{align*}
\Theta' & = z_\alpha' - z_\alpha d-\frac{\partial x}{\partial\alpha}d' \\
& = -Q_\alpha - (d+1)\left( z_\alpha - d\frac{\partial x}{\partial\alpha} \right) - \frac{\partial x}{\partial\alpha}(2f-M_0)\\
& = -Q_\alpha - \Theta (d+1) + \frac{\partial x}{\partial\alpha}(M_0-2f).
\end{align*}
Here, we used (\ref{closed}b), (\ref{closed}c) and \eqref{flowpath}.  Differentiating once again and using the expression for $\Theta$ along with (\ref{closed}d), (\ref{closed}a), (\ref{closed}c) and \eqref{flowpath}, 
\begin{align*}
\Theta'' &= -M_0z_\alpha  -\Theta' (d+1) -d' \Theta +z_\alpha (M_0-2f) +\frac{\partial x}{\partial\alpha}(-2f')\\
& = -\Theta' (d+1) + (d^2+d-2f+M_0)\Theta-2f\left(  z_\alpha - d\frac{\partial x}{\partial\alpha} \right) \\
& =- (d+1) \Theta' - (4f-M_0-d^2-d)\Theta. 
\end{align*}
Note that $\Theta(0)=0$ and $\Theta'(0)=0$, therefore, $\Theta\equiv 0$. So, $d=u_x$. 
In a similar manner we have 
$$
u_t +uu_x +u =-E(x, t). 
$$
We need to show $E(x, t)=\partial W* \rho$. In order to achieve this, it suffices to prove 
$$
q(\alpha, t)\equiv  0,
$$
where 
\[
q(t) 
= Q(\alpha,t) - \int_\mathbb{R}[-sgn(x(\alpha,t)-y(\beta,t))+(x(\alpha,t)-y(\beta,t)]\rho_0(\beta)\, d\beta.
\]
Observe that $q(0)=0$. Since the characteristics don't cross, we have  
$$
{\rm sgn}(x(\alpha,t)-x(\beta,t))={\rm sgn}(\alpha-\beta).
$$ 
This observation along with the characteristic flow equation gives,
\begin{align*}
q' & = Q'-\int_\mathbb{R}(z(\alpha,t)-z(\beta,t))\rho_0(\beta)\, d\beta \\
    & =M_0 z -M_1e^{-t} -z \int_\mathbb{R}\rho_0(\beta)\, d\beta +\int_\mathbb{R} z(\beta,t) \rho_0(\beta)\, d\beta \\
    & =\int_\mathbb{R} z(\beta,t) \rho_0(\beta)\, d\beta -M_1e^{-t}.
\end{align*}
We claim the right hand side is zero.  Hence $q\equiv 0$, we have arrived at equation (\ref{cc}b).

To prove the above claim,  we first show  $\int_\mathbb{R} Q(\alpha,t)\rho_0(\alpha)\, d\alpha=0$ which will be essential later. 
To see this, set
\[
a(t) := \int_\mathbb{R} Q(\alpha,t)\rho_0(\alpha)\, d\alpha.\]
One can check that $a(0)=0$. From (\ref{closed}d),
\[a'(t) = M_0\left( \int_\mathbb{R}z\rho_0\, d\alpha - M_1e^{-t} \right). 
\]
Observe that $a'(0)=0$. Differentiating once again, using the above expression for derivative and using (\ref{closed}b),
\[a'' + a' + M_0 a=0. \]
Consequently,
\[\int_\mathbb{R}Q(\alpha,t)\rho_0(\alpha)\, d\alpha = 0.\]
Multiplying (\ref{closed}b) by $\rho_0$, integrating and using the above result, we get
\begin{equation*}
\int_\mathbb{R} z(\alpha,t)\rho_0(\alpha)\, d\alpha = M_1 e^{-t}.
\end{equation*}
This proves the claim.  Finally, we show that the solution obtained for \eqref{cc} with the above construction indeed satisfies the designated boundary conditions. Since $f(\alpha,t)=\rho_0(\alpha)/(\partial x /\partial\alpha)$, we have $\rho\to 0$ as $|x|\to\infty$ with the same rate as $\rho_0$ since $\partial x/\partial \alpha$ has a strictly positive lower bound $\forall t\in[0,T]$. And since (\ref{closed}b) and (\ref{closed}d) form a closed linear system, we can explicitly solve for $z(\alpha,t)$ via 
\[z'' + z' + M_0 z = M_1e^{-t},\]
\[
z(0) = u_0(\alpha),\quad z'(0) = -u_0(\alpha) -E_0(\alpha).
\]
Since $u_{0x}$ is bounded and observing the $\alpha M_0$ term in $E_0(\alpha)$, both $z(0)$ and $z'(0)$ may grow at most linearly in $\alpha$, so does $z$. 
Hence, by hypothesis, 
$$
\lim_{|x|\to\infty}\rho u^2 =\lim_{|\alpha|\to\infty} (\partial\alpha/\partial x) \rho_0 z^2= 0.
$$ 
Finally, we show that for a fixed $\alpha \in\mathbb{R}$, for a finite time breakdown to occur, we must have $\lim_{t\to T^{*-}}d(t)=-\infty$. Otherwise, we would have  $\lim_{t\to T^{*-}}d(t)=\infty$.  Then  $\exists\epsilon >0$ such that for $t\in I:= \{ t: T^*-\epsilon<t<T^*\}$,
\[f' = -fd <0.\]
Therefore, there exists a constant $D>0$ such that $2f-M_0 <D$. Hence, from (\ref{closed}c),  $d'<D$;  Integration from some $t_0\in I$ gives  
\[
d(T^* ) < d(t_0)+D(T^*-t),
\]
which is a contradiction.  This completes the proof to Theorem \ref{l2}.

\section{Uniqueness of critical threshold curves}
\begin{lemma}
$Q_a$, $Q_b$, $Q_1$ and $Q_2$ in Theorem \ref{th2} are uniquely defined.
\end{lemma}
\begin{proof}
We will prove the existence and uniqueness for $Q_a$ only, and similar analysis can be carried out for $Q_b, Q_1, Q_2$. 
Consider an auxiliary problem of the form 
\[
\frac{d S}{dr} = \frac{r}{\nu r+k(1-c S )},\quad S(0)=0,  r>0.
\]
One can check the right hand side function is continuous and locally Lipschitz in $S$ around the origin.  Hence, a unique, strictly increasing solution $S$ exists on $[0,\delta]$ for some small $\delta>0$. We can show that $S^{-1}$ satisfies (\ref{qa}). Now, suppose $Q$ and $\widetilde{Q}$ are two solutions to (\ref{qa}). Note from (\ref{qa}) that any solution has to be strictly increasing in a neighbourhood for $s\geq 0$. Hence, $Q$ and $\widetilde{Q}$ have inverses, $Q^{-1}$ and $\widetilde{Q}^{-1}$. But the inverse being unique, we have $Q^{-1}=\widetilde{Q}^{-1}=S$. Hence, $Q_a$ is uniquely defined. 
\end{proof}


\section*{Acknowledgement}
This work was partially supported by the National Science Foundation under Grant DMS1812666 and by NSF Grant RNMS (Ki-Net) 1107291.

\bibliographystyle{abbrv}

\end{document}